\documentclass[11pt,leqno]{article}
\usepackage{graphicx, amsfonts, amsthm, amsxtra, amssymb, verbatim,rotating}
\usepackage[mathscr]{euscript}
\begin{document}

\begin{center}
{\LARGE\bf Heun equations coming from geometry \footnote{ 
Math. classification:  33E30, 33C05 
}}
\\
\vspace{.25in} {\large {\sc Hossein Movasati, Stefan Reiter}} \\
Instituto de Matem\'atica Pura e Aplicada, IMPA, \\
Estrada Dona Castorina, 110,\\
22460-320, Rio de Janeiro, RJ, Brazil, \\
E-mail:
{\tt hossein@impa.br, reiter@impa.br} 
\end{center}
%%%%%%%%%%%%%%%%%%%%%%%%%%%%%%%%%%%%%%%%%%%%%%%%%%%%%%%%%%%%%%%%%%%%%%%%%%%
\begin{abstract}
We give a list of Heun equations
which are Picard-Fuchs associated to families of
algebraic varieties. Our list is based on the classification of families of elliptic curves with four singular fibers done by Herfurtner. 
We also show that pull-backs of hypergeometric functions by rational Belyi functions with restricted ramification data give rise to  Heun equations.  
\end{abstract}
%---------------------------------------------------------------------
\newtheorem{theo}{Theorem}
\newtheorem{exam}{Example}
\newtheorem{coro}{Corollary}
\newtheorem{defi}{Definition}
\newtheorem{prob}{Problem}
\newtheorem{lemma}{Lemma}
\newtheorem{prop}{Proposition}
\newtheorem{rem}{Remark}
\newtheorem{cor}{Corollary}
\newtheorem{conj}{Conjecture}
%-----\newcommand\diff[1]{\frac{d #1}{dz}} %Differential operator
\def\Gal{{\rm Gal}}              %The Galois group
\def\Z{\mathbb{Z}}                   %Integer  numbers
\def\Q{\mathbb{Q}}                   %Rational  numbers
\def\C{\mathbb{C}}                   %Complex numbers
\def\ring{{\sf R}}                             %A ring 
\def\R{\mathbb{R}}                   %real numbers
\def\N{\mathbb{N}}                   %natural numbers
\def\uhp{{\mathbb H}}                %upper half plane
\newcommand{\mat}[4]{
     \begin{pmatrix}
            #1 & #2 \\
            #3 & #4
       \end{pmatrix}
    }                                %two by two matrices

\newcommand\SL[2]{{\rm SL}(#1, #2)}    %SL(2,Z)
\def\per{{\sf pm}}

\def\la{\lambda}
\def\th{\theta}
\def\P{\mathbb P}
%--------------------Stefan notation
\theoremstyle{plain}
\def\NN{{\mathbb N}}
\def\CC{{\mathbb C}}
\def\ZZ{{\mathbb Z}}
\def\QQ{{\mathbb Q}}
\def\SL{{\rm SL}}
\def\PSL{{\rm PSL}}
\def\GO{{\rm GO}}
\def\GL{{\rm GL}}
\def\PGL{{\rm PGL}}
\def\dz{{\rm d} z}
\def\dx{{\rm d} x}

\def\la{{\lambda}}
\def\nlambda{{\tilde{\lambda}}}
\def\nue{{\nu}}
\def\ele{{\rm L}}
\def\tparam{{t}}
\def\Mu{{\mu}}
\def\th{{\theta}}
\def\thh{{\tilde{\theta}}}
\def\Th{{\Theta}}
\def\tr{{\rm tr}}
\def\rk{{\rm rk}}
\def\Trace{{\rm Tr}}
\def\Mat{{\rm Mat}}
\def\al{{\alpha}}
\def\diag{{\rm diag}}
\def\id{{\rm id}}
\def\k{{\frak k}}
\def\l{{\frak  l}}
\def\Res{{\rm  Res}}
\def\dR{\rm dR}
%--------------------------------------------------

%------------------------------------------
\section{Introduction}
For a linear differential equation which depends on some parameters, it is a natural question to ask for which values of the parameters the specialized differential equation comes from geometry.  We say that the linear differential equation comes from geometry if there is a proper family of algebraic varieties $X\to \P^1$ over $\C$ and a differential form $\omega\in H^i_{\dR}(X/\P^1)$ such that the periods $\int_{\delta_z}\omega$, where $\delta_z\in H_i(X_z,\Z)$ is a continuous  family of cycles, span the solutions space of the linear differential equation.  Such linear differential equations are also called  Picard-Fuchs equations  (for further details see \cite{and88}, Chapter II, \S 1).

If a linear differential equation comes from geometry then it is well-known that the exponents of its singularities are all rational numbers (see for instance \cite{ka70} and the references therein). This implies that the Gauss hypergeometric equation with parameters $a,b,c$  comes from geometry if and only if the exponents of its singular set are rational numbers and hence if and only if $a,b,c$ are rationals. The next non-trivial family of linear differential equations is  the family of Heun equations:
\begin{equation}
 \label{1.09.08}
y''+ (\frac{1-\th_1}{z-t}+\frac{1-\th_2}{z}+\frac{1-\th_3}{z-1})y'+(\frac{\th_{41}\th_{42}z-q}{z(z-1)(z-t)})y=0
\end{equation}
with
$$
\th_{41}=-\frac{1}{2}(\th_1+\th_2+\th_3-2+\th_4),\ \th_{42}= -\frac{1}{2}( \th_1+\th_2+\th_3-2-\th_4).
$$
As we mentioned, if (\ref{1.09.08}) comes from geometry then the exponents $\th_i,\ i=1,2,\ldots,4,$  are rational numbers. Now, our problem reduces to the following one: For which rational numbers $\th_i,\ i=1,\ldots,4,$ and complex numbers $t,q\in\C$ does the corresponding equation (\ref{1.09.08}) come from geometry.

{
\begin{sidewaystable}[!htbp]
\tiny
\begin{center}
Table 1: Heun equations coming from geometry, $a,b,c\in\Q$
\begin{tabular}{|c|c|c|c|c|c|c|c|c|}
%\multicolumn{11}{c}{$***$}
\hline
*  &   $q$ & $t$ & $\th_1$ & $\th_2$ & $\th_3$ & $\th_4$  & $\th_{42}$ & $\th_{41}$  \\ \hline
\hline
1 &
% (27z4+108z3+162z2+84z+3) $ &
% (27z6+162z5+405z4+504z3+297z2+54z-1) $ &
$ %(6a^2-5a+\frac{2}{3}) 
\frac{1}{3}(3a-2)(6a-1)t_1 $ &
$\frac{t_1^2}{3}, \ t_1^2+3t_1+3=0 $ &
$  a-\frac{1}{2}  $ &
$  a-\frac{1}{2}  $ &
$  a-\frac{1}{2}  $ &
$  9a-\frac{9}{2}  $ &
$  3a-\frac{1}{2}  $ &
$  -6a+4  $

\\ \hline
2 &
%$ (48z4-48z2+3  $ 
%$ (64z6-96z4+30z2+1  $ &
$  0  $ &
$ -1 $ &
$  b-\frac{1}{2}  $ &
$  2b-1  $ &
$  b-\frac{1}{2}  $ &
$  4a+4b-4  $ &
$ 2a $ &
$  -2a-4b+4  $ \\ \hline
3 &
% (12z4-48z3+24z+12  $ &
% (8z6-48z5+48z4+56z3+12z2+24z+8  $ &
%$  -12ab+10a-24b^2+44b-20  $ &
$  -2(a+2b-2)(6b-5) $ &
$ -8 $ &
$  b-\frac{1}{2}  $ &
$  3b-\frac{3}{2}  $ &
$  a+b-1  $ &
$  3a+3b-3  $ &
$  a-b+1  $ &
$  -2a-4b+4  $ \\  \hline
4 &
% (3z4-36z3+42z2+36z+3  $ &
% (z6-18z5+75z4+75z2+18z+1  $ &
$  -3(10a-7)(3a-2)t_1%(-90a^2+123a-42)t_1 
$ &
$-t_1^2,\ t_1^2-11 t_1-1=0$ &
$  a-\frac{1}{2}  $ &
$  5a-\frac{5}{2}  $ &
$  a-\frac{1}{2}  $ &
$  5a-\frac{5}{2}  $ &
$  -a+\frac{3}{2}  $ &
$  -6a+4  $
\\ \hline
5 &
% (12z4-12z2+12  $ &
% (8z6-12z4-12z2+8  $ &
$  0  $ &
$ -1 $ &
$  a+c-1  $ &
$  2a+2b-2  $ &
$  a+c-1  $ &
$  2b+2c-2  $ &
$  -2a+2  $ &
$  -2a-2b-2c+4  $
\\ \hline
6 &
% (3z4+12z3+18z2+36z+27  $ &
% (-z6-6z5-15z4+45z2+54z+27  $ &
$  %(-6a^2+9a-\frac{10}{3})
\frac{-1}{3}(6a-5)(3a-2)t_1$ &
$\frac{t_1^2}{3}, t_1^2+3t_1+3=0$ &
$  3a-\frac{3}{2}  $ &
$  3a-\frac{3}{2}  $ &
$  3a-\frac{3}{2}  $ &
$  3a-\frac{3}{2}  $ &
$  -3a+\frac{5}{2}  $ &
$  -6a+4  $
\\ \hline
7 & 
% (12z4-72z3+180z2-144z  $ &
% (8z6-72z5+288z4-576z3+540z2-108z  $ &
%$ (-15077/3456a2+43223/10368a-127907/124416  t_1+ (81235/10368a2-216241/31104a+461893/373248  $ &
$ %(-\frac{64}{27}a^2+\frac{178}{81}a-\frac{100}{243})
   \frac{-2}{243}(96a-25)(3a-2)t_1 $ &
$\frac{t_1^2}{9},3t_1^2-14t_1+27=0$ &
$  a-\frac{1}{2}  $ &
$ \frac{1}{3} $ &
$  a-\frac{1}{2}  $ &
$  8a-4  $ &
$  3a-\frac{2}{3}  $ &
$  -5a+\frac{10}{3}  $
\\ \hline
8 &
% (108z4+432z3+504z2+168z  $ &
% (216z6+1296z5+2808z4+2664z3+1044z2+96z  $ &
$ %(-\frac{343}{32}a^2+\frac{835}{96}a-\frac{149}{144})
\frac{-1}{288}(3a-2)(1029a-149) $ &
$ \frac{81}{32} $ &
$  a-\frac{1}{2}  $ &
$ \frac{1}{3} $ &
$  2a-1  $ &
$  7a-\frac{7}{2}  $ &
$  2a-\frac{1}{6}  $ &
$  -5a+\frac{10}{3}  $
\\ \hline
9 &
% (192z4+2064z3+2043z2-1824z+192  $ &
% (512z6+8256z5+30360z4-3125z3+25500z2-7296z+512  $ &
%-250a^2+\frac{2125}{6}a-125
$\frac{-125}{6}(4a-3)(3a-2)  $ &
$ -80 $ &
$  a-\frac{1}{2}  $ &
$  4a-2  $ &
$ \frac{1}{3} $ &
$  5a-\frac{5}{2}  $ &
$ \frac{5}{6} $ &
$  -5a+\frac{10}{3}  $
\\ \hline
10 &
% (243z4-756z3-78z2+396z+243  $ &
% (729z6-3402z5+2295z4+3700z3+375z2+1782z+729  $ &
$   %-25a^2+\frac{75}{2}a-\frac{125}{9} 
\frac{-25}{18}(3a-2)(6a-5)$ &
$ -\frac{27}{5} $ &
$ \frac{1}{3} $ &
$  3a-\frac{3}{2}  $ &
$  2a-1  $ &
$  5a-\frac{5}{2}  $ &
$ \frac{5}{6} $ &
$  -5a+\frac{10}{3}  $
\\ \hline
11 &
%  12z4+48z3+120z2+72z  $ &
%  8z6+48z5+168z4+280z3+252z2  $ &
$ %(\frac{147}{128}a^2-\frac{67}{64}a+\frac{3}{16})
\frac{1}{128}(49a-12)(3a-2)t_1 $ &
$\frac{t_1^2}{8}, 4t_1^2+13t_1+32=0$ &
$  a-\frac{1}{2}  $ &
$ \frac{1}{2} $ &
$  a-\frac{1}{2}  $ &
$  7a-\frac{7}{2}  $ &
$  \frac{5}{2}a-\frac{1}{2}  $ &
$  -\frac{9}{2}a+3  $
\\ \hline
12 & 
%  12z4-72z3+108z2-36z  $ &
%  8z6-72z5+216z4-252z3+108z2  $ &
$  %-\frac{9}{16}a^2-\frac{9}{8}ab+\frac{9}{8}a
\frac{-9}{16}a(a+2b-2)$ &
$ \frac{1}{4} $ &
$  2b-1  $ &
$ \frac{1}{2} $ &
$  b-\frac{1}{2}  $ &
$  3a+3b-3  $ &
$ \frac{3}{2}a $ &
$  -\frac{3}{2}a-3b+3  $
\\ \hline
13 &
%  1875z4+16500z3+11250z2-12300z+1875  $ &
%  15625z6+206250z5+594375z4-200500z3+392775z2-153750z+15625  $ &
$  %\frac{351}{250}a^2-\frac{1053}{500}a+\frac{39}{50}
 \frac{39}{500}(3a-2)(6a-5)  $ &
$ -\frac{3}{125} $ &
$ \frac{1}{2} $ &
$  3a-\frac{3}{2}  $ &
$  a-\frac{1}{2}  $ &
$  5a-\frac{5}{2}  $ &
$  \frac{1}{2}a+\frac{1}{2}  $ &
$  -\frac{9}{2}a+3  $
\\ \hline
14 &
%  48z4-48z3-9z2+6z+3  $ &
%  64z6-96z5+6z4+25z3-3z2+3z+1  $ &
$  %-\frac{9}{2}ab+\frac{15}{4}a-9b^2+\frac{33}{2}b-\frac{15}{2} 
\frac{-3}{4}(a+2b-2)(6b-5)$ &
$ -3 $ &
$ \frac{1}{2} $ &
$  3b-\frac{3}{2}  $ &
$  a+b-1  $ &
$  2a+2b-2  $ &
$  \frac{1}{2}a-b+1    $ &
$ -\frac{3}{2}a-3b+3  $
% %  3z4+6z3-9z2-48z+48  $ &
% %  z6+3z5-3z4+25z3+6z2-96z+64  $ &
% $  27/16a2+27/8ab-81/16a-27/8b+27/8  $ &
% $ -1/3 $ &
% $ 1/2 $ &
% $  2a+2b-2  $ &
% $  a+b-1  $ &
% $  3b-5/2  $ &
% $  -3/2a+3/2  $ &
% $  -3/2a-3b+3  $
\\ \hline
15 &
%  27z4-24z2  $ &
%  27z6-36z4+8z2  $ &
$  0  $ &
$ -1 $ &
$  a-\frac{1}{2}  $ &
$ \frac{2}{3} $ &
$  a-\frac{1}{2}  $ &
$  6a-3  $ &
$  2a-\frac{1}{3}  $ &
$  -4a+\frac{8}{3}   $
% %  -24z2+27  $ &
% %  8z4-36z2+27  $ &
% $  8a2-12a+40/9  $ &
% $ -1 $ &
% $  a-1/2  $ &
% $  6a-3  $ &
% $  a-1/2  $ &
% $ -1/3 $ &
% $  -4a+10/3  $ &
% $  -4a+8/3  $
\\ \hline
16 &
%  12z4+96z3+120z2  $ &
%  8z6+96z5+312z4+224z3+108z2  $ &
$  -\frac{14}{3}a+\frac{28}{9}  $ &
$ \frac{27}{2} $ &
$  a-\frac{1}{2}  $ &
$ \frac{2}{3} $ &
$  2a-1  $ &
$  5a-\frac{5}{2}  $ &
$  a+\frac{1}{6}  $ &
$  -4a+\frac{8}{3}   $
% %  120z2+96z+12  $ &
% %  108z4+224z3+312z2+96z+8  $ &
% $  -2808/25a2+756/5a-1272/25  $ &
% $ 2/27 $ &
% $  a-1/2  $ &
% $  5a-5/2  $ &
% $  2a-1  $ &
% $ -1/3 $ &
% $  -4a+10/3  $ &
% $  -4a+8/3  $
\\ \hline
17 & 
%  27z4-12z3-30z2-12z+27  $ &
%  27z6-18z5-43z4+68z3-43z2-18z+27  $ &
$ %-4a^2+6a-\frac{20}{9} 
\frac{-2}{9}(3a-2)(6a-5) $ &
$ -1 $ &
$ \frac{2}{3} $ &
$  3a-\frac{3}{2}  $ &
$  2a-1  $ &
$  3a-\frac{3}{2}  $ &
$  -a+\frac{7}{6}   $ &
$  -4a+\frac{8}{3} $
% %  27z4-12z3-30z2-12z+27  $ &
% %  27z6-18z5-43z4+68z3-43z2-18z+27  $ &
% $  -2/3a+4/9  $ &
% $ -1 $ &
% $ 2/3 $ &
% $  3a-3/2  $ &
% $  2a-1  $ &
% $  3a-5/2  $ &
% $  -a+7/6  $ &
% $  -4a+8/3  $
\\ \hline
18 &
%  3z4-54z3+345z2-774z+147  $ &
%  z6-27z5+294z4-1575z3+3990z2-3339z-343  $ &
$ %(-\frac{58}{49}a^2+\frac{23}{21}a-\frac{10}{49})
\frac{-1}{147}(58a-15)(3a-2)t_1 $ &
$\frac{t_1^2}{49},\ t_1^2-13 t_1+49=0$ &
$ \frac{1}{3} $ &
$  a-\frac{1}{2}  $ &
$ \frac{1}{3} $ &
$  7a-\frac{7}{2}  $ &
$  3a-\frac{5}{6}  $ &
$  -4a+\frac{8}{3}  $
% %  147z4-774z3+345z2-54z+3  $ &
% %  -343z6-3339z5+3990z4-1575z3+294z2-27z+1  $ &
% $  686/27a2-61040/999a+438145/17982  t_1+  -280/27a2+101257/6993a-3772733/881118  $ &
% $49t_12, 1-13t_1+49t_12$ &
% $ 1/3 $ &
% $  7a-7/2  $ &
% $ 1/3 $ &
% $  a-3/2  $ &
% $  -3a+13/6  $ &
% $  -4a+8/3  $
\\ \hline
19 &
%  27z4-30z2+3  $ &
%  27z6-45z4+17z2+1  $ &
$ 0  $ &
$ -1 $ &
$\frac{ 1}{3} $ &
$  2a-1  $ &
$ \frac{1}{3} $ &
$  6a-3  $ &
$  2a-\frac{1}{3}  $ &
$  -4a+\frac{8}{3}  $
\\ \hline
20 & 
%  7776z2-3888z / z3+12z2+48z+64  $ &
%  11664z5+513216z4+1026432z2-46656z / z6+24z5+240z4+1280z3+3840z2+6144z+4096  $ &
$%(-16a2+68/3a-8)
\frac{-4}{3}(4a-3)(3a-2)t_1 $ &
$-\frac{t_1^2}{2}, t_1^2-10t_1-2$ &
$  4a-2  $ &
$ \frac{1}{3} $ &
$  4a-2  $ &
$ \frac{1}{3} $ &
$  -4a+3  $ &
$  -4a+\frac{8}{3}  $
% %  -972z2+1944z / 64z3+48z2+12z+1  $ &
% %  -5832z5+128304z4+64152z2+1458z / 4096z6+6144z5+3840z4+1280z3+240z2+24z+1  $ &
% $  10024/297a2-14704/297a+3593/198  t_1+  51520/297a2-75640/297a+9250/99  $ &
% $-2t_12,  6t_12+30t_1-3 $ &
% $  4a-2  $ &
% $ 1/3 $ &
% $  4a-2  $ &
% $ -2/3 $ &
% $  -4a+3  $ &
% $  -4a+8/3  $
\\ \hline
21 & 
% $ 2*z* z-1 * 6*z2+6*tn*z- 3+tn  ;$ &
% $ 2*z2* z-1 * 4*z3-2* 1-3*tn *z2-4* 2+tn *z+ 5-tn  $ &
$ %(-\frac{27}{2}\zeta-\frac{29}{4})a^2+(\frac{25}{2}\zeta+\frac{20}{3})a+(-\frac{7}{3}\zeta-\frac{11}{9}) 
(\frac{-27}{2}\zeta-\frac{29}{4})(a-\frac{10}{9589}\zeta-\frac{7442}{28767})(a-\frac{2}{3})
$ & 
$ -\frac{2}{7}  (3 \zeta+1),\ \zeta^2+3=0 $ &
$a-\frac{1}{2} $ &
$ \frac{1}{2} $ &
$\frac{1}{3} $ &
$ 6a-3 $ &
$  \frac{5}{2}a-\frac{2}{3}  $ &
$ -\frac{7}{2}a+\frac{7}{3}$ 
\\ \hline
22 &
%  12z4-372z3+3675z2-11340z  $ &
%  8z6-372z5+6558z4-53407z3+188838z2-182250z  $ &
$  %-\frac{686}{125}a^2+\frac{112}{25}a-\frac{616}{1125}
\frac{-14}{1125}(3a-2)(147a-22)  $ &
$ \frac{189}{125} $ &
$\frac{ 1}{2} $ &
$ \frac{1}{3} $ &
$  2a-1  $ &
$  5a-\frac{5}{2}  $ &
$ \frac{3}{2}a-\frac{1}{6} $ &
$ -\frac{7}{2}a+\frac{7}{3} $
% %  -11340z3+3675z2-372z+12  $ &
% %  -182250z5+188838z4-53407z3+6558z2-372z+8  $ &
% $  -11375/2304a2+56875/6912a-11375/3456  $ &
% $ 189/125 $ &
% $  2a-1  $ &
% $  5a-5/2  $ &
% $ 1/2 $ &
% $ -2/3 $ &
% $  -7/2a+8/3  $ &
% $  -7/2a+7/3  $
\\ \hline
23 &
%  48z4+1488z3+4215z2-25434z+19683  $ &
%  64z6+2976z5+31494z4-39367z3+503469z2-1030077z+531441  $ &
$  %\frac{77}{54}a^2-\frac{77}{36}a+\frac{385}{486}
\frac{77}{972}(3a-2)(6a-5) $ &
$ -\frac{1}{27 } $ &
$ \frac{1}{2} $ &
$  3a-\frac{3}{2}  $ &
$ \frac{1}{3} $ &
$  4a-2  $ &
$ \frac{1}{2}a+\frac{1}{3} $ &
$ -\frac{7}{2}a+\frac{7}{3} $
% %  19683z4-25434z3+4215z2+1488z+48  $ &
% %  531441z6-1030077z5+503469z4-39367z3+31494z2+2976z+64  $ &
% $  -7/144a2+5/432a+1/72  $ &
% $ -27 $ &
% $ 1/2 $ &
% $  4a-2  $ &
% $ 1/3 $ &
% $  3a-5/2  $ &
% $  -1/2a+5/6  $ &
% $  -7/2a+7/3  $
\\ \hline
24 &
%  3z4+6z3-45z2  $ &
%  z6+3z5-21z4-23z3+96z2  $ &
$  -\frac{1}{6}a+\frac{1}{9}  $ &
$ -\frac{16}{9} $ &
$  a-\frac{1}{2}  $ &
$ \frac{2}{3} $ &
$ \frac{1}{3} $ &
$  5a-\frac{5}{2}  $ &
$ 2a-\frac{1}{2} $ &
$ -3a+2 $
% %  -45z2+6z+3  $ &
% %  96z4-23z3-21z2+3z+1  $ &
% $  304/25a2-1328/75a+32/5  $ &
% $ -9/16 $ &
% $  a-1/2  $ &
% $  5a-5/2  $ &
% $ 1/3 $ &
% $ -1/3 $ &
% $  -3a+8/3  $ &
% $  -3a+2  $
\\ \hline
25 &
%  108z4-144z3+36z2  $ &
%  216z6-432z5+252z4-40z3+4z2  $ &
$  -3a+2 $ &
$ 9 $ &
$ \frac{1}{3} $ &
$\frac{ 2}{3} $ &
$  2a-1  $ &
$  4a-2  $ &
$ a $ &
$ -3a+2 $
% %  36z2-144z+108  $ &
% %  4z4-40z3+252z2-432z+216  $ &
% $  -243/8a2+333/8a-57/4  $ &
% $ 1/9 $ &
% $ 1/3 $ &
% $  4a-2  $ &
% $  2a-1  $ &
% $ -1/3 $ &
% $  -3a+8/3  $ &
% $  -3a+2  $
\\ \hline
26 &
%  3z4-48z3+525/2z2-510z+1875/16  $ &
%  z6-24z5+909/4z4-1049z3+36891/16z2-7125/4z-15625/64  $ &
$ %(-\frac{114}{125}a^2+\frac{103}{125}a-\frac{18}{125})
\frac{-1}{125}(3a-2)(38a-9)t_1 $ &
$\frac{4t_1^2}{125},\ t_1^2-11t_1+125/4=0$ &
$ \frac{1}{2} $ &
$  a-\frac{1}{2}  $ &
$ \frac{1}{2} $ &
$  5a-\frac{5}{2}  $ &
$  2a-\frac{1}{2}  $ &
$  -3a+2  $
\\ \hline
27 & 
%  12z4-15z2+3  $ &
%  8z6-15z4+6z2+1  $ &
$  0 $ &
$ -1 $ &
$ \frac{1}{2} $ &
$  2b-1  $ &
$ \frac{1}{2} $ &
$  2a+2b-2  $ &
$ a $ &
$ -a-2b+2 $
% %  3z4-15z2+12  $ &
% %  z6+6z4-15z2+8  $ &
% $  1/2a2+ab-3/2a-b+1  $ &
% $ -1 $ &
% $ 1/2 $ &
% $  2a+2b-2  $ &
% $ 1/2 $ &
% $  2b-2  $ &
% $  -a+1  $ &
% $  -a-2b+2  $
\\ \hline
28 &
%  3z3+18z2-9z  $ &
%  6z4+18z2  $ &
$ %(-3a^2+\frac{9}{2}a-\frac{5}{3})
\frac{-1}{6}(6a-5)(3a-2)t_1$ &
$-\frac{t_1^2}{3},\ t_1^2-6t_1-3=0$ &
$  3a-\frac{3}{2}  $ &
$ \frac{1}{2} $ &
$  3a-\frac{3}{2}  $ &
$ \frac{1}{2} $ &
$  -3a+\frac{5}{2}  $ &
$  -3a+2  $
\\ \hline
29 &
%  12z4+48z3-60z2  $ &
%  8z6+48z5-12z4-152z3+108z2  $ &
$  \frac{5}{162}a-\frac{5}{243} $ &
$ -\frac{5}{27} $ &
$ \frac{1}{2} $ &
$ \frac{2}{3} $ &
$  a-\frac{1}{2}  $ &
$  4a-2  $ &
$  \frac{3}{2}a-\frac{1}{3}  $ &
$  -\frac{5}{2}a+\frac{5}{3}  $
\\ \hline
30 & 
%  27z4-42z3+15z2  $ &
%  27z6-63z5+47z4-13z3+2z2  $ &
$   -\frac{5}{3}a+\frac{10}{9}$ &
$ 5 $ &
$ \frac{1}{2} $ &
$ \frac{2}{3} $ &
$  2a-1  $ &
$  3a-\frac{3}{2}  $ &
$  \frac{1}{2}a+\frac{1}{6}  $ &
$  -\frac{5}{2}a+\frac{5}{3}  $
\\ \hline
31 &
%  3z4-6z2+3  $ &
%  z6-z4-z2+1  $ &
$  0 $ &
$ -1 $ &
$\frac{ 2}{3} $ &
$  2a-1  $ &
$ \frac{2}{3} $ &
$  2a-1  $ &
$ \frac{1}{3} $ &
$  -2a+\frac{4}{3}  $
% \\ \hline
% $  0 $ &
% $ -1 $ &
% $(1-\gamma)2 $ &
% $ 2\gamma-4\alpha-1 $ &
% $ (1-\gamma)2 $ &
% $  2\gamma-4\alpha-1  $ &
% $ 2 \gamma-1 $ &
% $  2 \alpha  $
% %\be=\ga-\al-1/2
\\ \hline
32 & 
%  3z4-3z2  $ &
%  z6-2z4+z2  $ &
$  0  $ &
$ -1 $ &
$ \frac{1}{2} $ &
$ \frac{2}{3} $ &
$ \frac{1}{2} $ &
$  2a-1  $ &
$  a-\frac{1}{3}  $ &
$  -a+\frac{2}{3}  $
\\ \hline
33 &
% (3z4+12z3+18z2+9z) $ &
% (z6+6z5+15z4+18z3+9z2) $ &
$ %(\frac{3}{4}a^2-\frac{3}{4}a+\frac{1}{6})
\frac{1}{12}(3a-1)(3a-2)t_1$ &
$\frac{t_1^2}{3},\ t_1^2+3t_1+3=0$ &
$ \frac{1}{2}  $ &
$\frac{1}{2}   $ &
$\frac{1}{2}  $ &
$ 3a-\frac{3}{2} $ &
$ \frac{3}{2}a-\frac{1}{2} $ &
$ -\frac{3}{2}a+1 $
\\ \hline
34 &
%  3z4+6z3-9z2  $ &
%  z6+3z5-3z4-7z3+6z2  $ &
$  0 $ &
$ -\frac{1}{3} $ &
$ \frac{1}{2} $ &
$ \frac{2}{3} $ &
$ \frac{1}{3} $ &
$  3a-\frac{3}{2}  $ &
$  \frac{3}{2}a-\frac{1}{2}  $ &
$  -\frac{3}{2}a+1  $
\\ \hline
35 &
%  12z4-12z2  $ &
%  8z6-12z4+4z2  $ &
$  0  $ &
$ -1 $ &
$ \frac{1}{3} $ &
$\frac{ 2}{3} $ &
$ \frac{1}{3} $ &
$  4a-2  $ &
$  2a-\frac{2}{3}  $ &
$  -2a+\frac{4}{3}  $
\\ \hline 
36 &
%  3z4-42z3+201z2-324z  $ &
%  z6-21z5+174z4-694z3+1269z2-729z  $ &
$ %(-\frac{16}{27}a^2+\frac{16}{27}a-\frac{32}{243})
\frac{-16}{243}(3a-1)(3a-2)t_1 $ &
$\frac{t_1^2}{27},\ t_1^2-10t_1+27=0$ &
$ \frac{1}{2} $ &
$ \frac{1}{3} $ &
$ \frac{1}{2 }$ &
$  4a-2  $ &
$  2a-\frac{2}{3}  $ &
$  -2a+\frac{4}{3}  $
\\ \hline
37 &
%  3z4+42z3+291z2+576z  $ &
%  z6+21z5+219z4+1135z3+2880z2  $ &
$ %(\frac{75}{256}a^2-\frac{75}{256}a+ \frac{25}{384})
\frac{25}{768}(3a-2)(3a-1)t_1 $ &
$\frac{t_1^2}{64},t_1^2+11 t_1+64=0$ &
$ \frac{1}{3} $ &
$\frac{ 1}{2} $ &
$ \frac{1}{3} $ &
$  5a-\frac{5}{2}  $ &
$  \frac{5}{2}a-\frac{5}{6}  $ &
$  -\frac{5}{2}a+\frac{5}{3}  $
\\ \hline
38 &
%  12z4+48z3+72z2+36z  $ &
%  8z6+48z5+120z4+148z3+84z2+12z  $ &
$  %(3a^2-3a+\frac{2}{3})
\frac{1}{3}(3a-2)(3a-1)t_1  $ &
$\frac{t_1^2}{3},\ t_1^2+3t_1+3=0$ &
$ \frac{1}{3} $ &
$\frac{ 1}{3} $ &
$\frac{ 1}{3} $ &
$  6a-3  $ &
$  3a-1  $ &
$  -3a+2 $
\\ \hline
\end{tabular}
\end{center}
\end{sidewaystable}
}

We have observed that in the classification of families of elliptic curves with exactly four singular fibers (see \cite{her91})
 only $38$ of the $50$ examples give  us Heun equations (five of them give us linear differential equations associated to Painlev\'e VI equations with algebraic solutions, see \cite{dor01, ben05}, seven  of them can be reduced to families with three singular fibers by means of quadratic twists). Using this we have obtained a table of Heun equations coming from geometry, see Table 1 for $a,b,c\in\Q$. 
Table 1 contains the previously calculated list of Heun equations by R. S. Maier in \cite{Maier05}.

One application of  Table 1 can be found in \cite{R08}, where
the second author shows that one gets a list of Lam\'e equations with  arithmetic Fuchsian monodromy group applying the inverse Halphen transform 
to some of the examples.  In this way one obtains all those Lam\'e equations
where the  quaternion algebra %$A$ over $k$ 
associated 
to the arithmetic Fuchsian group is defined %a  quaternion algebra $A$ 
over $\mathbb{Q}$.
Thus one can relate Krammer's example in \cite{Krammer96}, that was considered
to be a counter example to a conjecture of Dwork, to a Gauss hypergeometric equation
via geometric operations.

In \S \ref{algorithm} we explain how to compute Table 1 using the Weierstrass form of families of elliptic curves with four singular fibers. 
The corresponding algorithms are implemented in the library {\tt painleve-heun.lib} in Singular, see \cite{GPS01}. Table 1 can be also computed  using the $j$-function of the corresponding family of elliptic curves. We explain this in \S \ref{legendre}. In \S \ref{belyisection} we state Theorem \ref{Riemann} which characterizes 
pull-backs of hypergeometric functions by rational Belyi functions. In
particular we get further Heun equations under restricted ramification data for the
Belyi functions. Since the $j$-invariants of the mentioned 38 examples in
\cite{her91} are Belyi functions, this method explains why we get Table 1. In
\S \ref{lame} we have derived Table 2 of Lam\'e equations,
i.e. $\th_1=\th_2=\th_3=\frac{1}{2}$, from Table 1. In \S
\ref{comparison} we compare Table 1 with examples we found in the literature. 

The authors thank the anonymous referees for their valuable
comments.

\section{Calculating Table 1 using the Weierstrass form}
\label{algorithm}

In this section we explain how we have obtained Table 1 using the Weierstrass form of elliptic curves. Despite the fact that 
the $j$-invariants of
the Herfurtner's list are special Belyi maps, the advantage of
this method is that for each item in the table it gives an explicit family of Riemann surfaces with four singular fibers. This can be useful
for arithmetic applications of Heun equations using the geometry of curves.

We take a family of elliptic curves
$$
y^2=f(x),\ f(x):=4x^3-g_2x-g_3,\ g_2,g_3\in\C(z)
$$
with four singular fibers. There are $50$ examples of such families which are listed by Herfurtner in \cite{her91}.
In the next step we check whether the polynomial $f(x)$ factorizes over
$\C(z)$. If $f(x)$ is a product of degree $2$ and degree $1$ polynomials then we redefine $g_2$ and $g_3$ in the following way
$$
f(x)=(4x^2-g_2x+g_3)(x+\frac{g_2}{4}).\ 
$$
If $f(x)$ is a product of three degree $1$  polynomials then we redefine $g_2$ and $g_3$ in the following way:
$$
f(x)=(4x+g_2+g_3)(x-\frac{g_2}{4})(x-\frac{g_3}{4})
$$
Corresponding to the above three cases we consider the following family of transcendent curves:
$$
y=(4x^3-g_2x-g_3)^a,\ 
$$
$$
y=(x+\frac{g_2}{4})^a(4x^2-g_2x+g_3)^b, 
$$
$$
y=(4x+g_2+g_3)^a(x-\frac{g_2}{4})^b(x-\frac{g_3}{4})^c 
$$
$$
a,b,c\in\C.
$$
One can recover the family of elliptic curves by setting $a=b=c=\frac{1}{2}$.
The corresponding systems in the variables $g_2$ and $g_3$ can be calculated from the system in three variables $t_1,t_2,t_3$ 
$$
dY=AY
$$ 
where
\begin{equation}
\label{A=}
A=
\end{equation}
{\tiny
$$
\frac{1}{(t_1-t_2)(t_1-t_3)}
\mat
{
\frac{1}{2}(b+c-2)t_1+ \frac{1}{2}(a+c-1)t_2+\frac{1}{2}(a+b-1)t_3 
}
{
-a-b-c+2
}
{
at_2t_3+(b-1)t_1t_3+(c-1)t_1t_2
}
{
-\frac{1}{2}(b+c-2)t_1-\frac{1}{2}(a+c-1)t_2
-\frac{1}{2}(a+b-1)t_3
}dt_1
$$
}
$$
+(\cdots)dt_2+(\cdots)dt_3
$$
and the matrix coefficient of $dt_2$ (resp $dt_3$) is obtained by permutation of $t_1$ with $t_2$ 
and $a$ with $b$ (resp. $t_1$ with $t_3$ and $a$ with $c$) in the matrix coefficient of $dt_1$ written above. This system is associated to the family of transcendental curves
$$
y=(t_1-t_3)^{\frac{1}{2}(1-a-c)}(t_1-t_2)^{\frac{1}{2}(1-a-b)}(t_2-t_3)^{\frac{1}{2}(1-b-c)}
(x-t_1)^a(x-t_2)^b(x-t_3)^c.
$$
For further details and explicit formulas for the three cases above see \cite{mov08}.  In this way, we calculate the linear differential equation satisfied by integrals $\int\frac{dx}{y}$, namely 
\begin{eqnarray*}\label{d.e.} y''+p_1(z) y'+p_2(z)y=0\end{eqnarray*}
and then we write it in  the $\SL$-form. 
The $\SL$-form of the above second order Fuchsian differential equation is by definition 
\begin{eqnarray*} \label{SL} 
y'' =p(z)y,& p(z)=-p_2(z)+\frac{1}{4}p_1(z)^2+\frac{1}{2} p_1'(z).\end{eqnarray*}
In the 50 families of elliptic curves in \cite{her91} there are seven families of elliptic curves with $I_0^*$ singularity. The corresponding singularity, namely  $\rho_4$ which is an arbitrary parameter, does not appear as a singularity of the $\SL$-form. Five other families depend on an extra parameter $\alpha$ and the corresponding SL-form has an apparent singularity. They give us algebraic solutions of the Painlev\'e VI equation and they are discussed in detail in 
\cite{ben05, dor01, more}. Therefore, the first twelve families in \cite{her91} do not yield  Heun differential equations. The next 38 families give us Heun equations in the $\SL$-form:
 \begin{eqnarray*} y''=p(z)y,&p(z)=
\frac{a_1}{(z-t)^2} + \frac{a_2}{z^2}+ \frac{a_3}{(z-1)^2}+\frac{a_4}{z(z-1)}+
               \frac{L}{z(z-t)(z-1)},
                \end{eqnarray*}
where
$$
a_4=-\frac{1}{4}(\sum_{i=1}^3\th_i^2-(\th_4+1)^2)+\frac{1}{2},
$$
$$
L=q-t\th_{41}\th_{42}+ \frac{(1-\th_1)}{2} ((1-\th_2)(t-1)+(1-\th_3)t).%\frac{(\th_1-1)(t\th_2+t\th_3-2t-\th_2+1)}{2t(t-1)}.
$$
Now it is just a matter of calculation to obtain the corresponding parameters from the SL-form. Our numbering row 1 till 38 in Table 1 corresponds to the 13th till  50th family in \cite{her91}. 

Among the 38 examples  there are 13 examples with two Galois conjugate singularities and with $g_2,g_3\in\Q(z)$. Since in Singular, see   \cite{GPS01},  we were not able to calculate in a ring with many transcendental and algebraic parameters, we have used the $\SL$-form with  singularities  $t_1,t_2=0,t_3$ and  $\infty$:
\begin{eqnarray*} 
p(z)=\sum_{i=1}^3 \frac{a_i}{(z-t_i)^2} +\frac{\tilde{a}_4}{z(z-t_3)}+
               \frac{t_1(t_1-t_3)/t_3 \cdot \tilde{L}}{(z-t_1)z(z-t_3)}
.\end{eqnarray*}
We have to treat the 21th example in a especial way because it is the only example in which $g_2$ and $g_3$ are not defined over $\Q(z)$. The corresponding sequence of commands in Singular are implemented in the library {\tt painleve-heun.lib}. This and the 38 families in \cite{her91}  can be downloaded from the first author's webpage. 

%---------------------------------------------------------------------------------------------------------

%-----------------------------------------

\section{Calculating Table 1 using the $j$-invariant}

\label{legendre}

In \cite{her91} 
Herfurtner has classified elliptic surfaces with four singular fibres in
Weierstrass form.
To each elliptic surface it corresponds a period,  a complete elliptic integral
of the first kind, depending on a parameter.
Thus it satisfies a Picard-Fuchs equation with regular singular points.
%corresponding to the singular fibres.
In 38 cases it is a Heun equation. All those equations are pull-backs
of the Gauss hypergeometric  equation $L$, where $L$ is the
uniformizing differential equation for $\PSL_2(\ZZ)$, 
by the $j$-invariant of the elliptic
curve, as already noted by Stiller studying classical uniformization problems
in \cite{Sti83}.
In 27 of the 38 cases  Doran showed that the Picard-Fuchs equation
is an  orbifold uniformizing differential equation,
see Chapter 4 in \cite{Doran98}.

The idea is  to replace the Picard-Fuchs equation $L$ satisfied by elliptic
integrals with suitable geometric Gauss hypergeometric equations satisfied
by abelian integrals
to obtain the one parameter families of geometric Heun equations in Table~1.
In Herfurtner's list
we find the following data:
The family of elliptic curves 
\[ y^2=4x^3-g_2(z)x-g_3(z),\; g_2(z), g_3(z) \in \CC(z), \]
the discriminant $\Delta=g_2^3-27 g_3^2$ and the
$j$-invariant $j=\frac{g_2^3}{g_2^3-27g_3^2}$.
It is easy to check that in the cases we consider, namely
 $I_1 \;I_1\; I_1\;  I_9-I_6\; II\; II \;II$ in \cite{her91}
(i.e. in Table 1, row 1-38) 
the $j$-invariant ramifies only at $0, 1$ and $\infty$. 
Such a function  is  also called a rational Belyi-function. 
In our cases the ramification indices   at $0$ are at most $3$
and at $1$ are at most $2$. Therefore, we will consider the pull-back of the hypergeometric function
\[ {}_2F_1(\alpha,\beta, \frac{2}{3},z),\quad \alpha=-\frac{a}{2}+\frac{1}{3},
\quad \alpha-\beta=-a+\frac{1}{2}\] 
with $j(z)$. Since the  Riemann scheme of the corresponding hypergeometric differential equation is
\[  \left( \begin{array}{cccc}
      0 & 1 &\infty \\
     0 & 0 &-\frac{a}{2}+\frac{1}{3} \\
 \frac{1}{3}& \frac{1}{2} &-\frac{1}{6}+\frac{a}{2}
   \end{array}\right) \]
this pull-back will satisfy (after a multiplication with an algebraic
function) a Heun equation depending on the parameter $a$.
We demonstrate this claim
via the following example, Table 1, row 7, ( $I_1I_1I_8II$ in Herfurtner's list).
First we recall  two basic transformations of  second order differential equations, which are readily to check:
\begin{rem}\rm\label{trafo}
Let $Y(z)$ be a solution of 
\begin{eqnarray*}
y''+p_1(z)y'+p_2(z)y=0.
\end{eqnarray*}

\begin{enumerate}
\item[a)] Then $Y(j(z))$ satisfies
\begin{eqnarray}\label{pb}  y''+ (p_1(j(z)) j'(z)-\frac{j''(z)}{j'(z)})y'+p_2(j(z)) j'(z)^2 y=0 \end{eqnarray}
\item[b)]  and $f(z)Y(z)$  satisfies
\begin{eqnarray}
y''+(p_1(z)-2\frac{f'(z)}{f(z)}) y' +(p_2(z)+\frac{2f'(z)^2}{f(z)^2}-
  \frac{p_1(z)f'(z)}{f(z)}- \frac{f''(z)}{f(z)}) y=0.
\end{eqnarray}
\end{enumerate}
\end{rem}

\begin{exam}\rm
Let $y^2=4x^3-g_2(z)x-g_3(z)$, where
\begin{eqnarray*}   g_2(z)=12 z (z^3-6 z^2+15 z-12),&& g_3(z)=4 z (2 z^5-18 z^4+72 z^3-144 z^2+135z-27)\\
\  j(z)=\frac{g_2^3}{g_2^3-27g_3^2}&=&-\frac{z(z^3-6z^2+15z-12)^3}{(3z^2-14z+27)}\\
  j(z)-1&=&-\frac{(2z^5-18z^4+72z^3-144z^2+135z-27)^2}{(3z^2-14z+27)}.\end{eqnarray*}
Thus  the ramification data is therefore given by the cycle decomposition
\begin{eqnarray*}  (3)(3)(3)(1),\quad (2)(2)(2)(2)(2), \quad (8)(1)(1).\end{eqnarray*}
The Hurwitz formula implies that the $j$-invariant is unramified outside
$0,1$ and $\infty$. Hence it is a Belyi-function.
Since
a hypergeometric function ${}_2F_1(\alpha,\beta, \gamma,z)$ satisfies
\begin{eqnarray*}  y''+p_1(z)y'+p_2(z)y=0,& \ \ p_1(z)=\frac{\gamma-(\alpha+\beta +1)z}{z (1-z)}, \ \ p_2(z)= \frac{\alpha \beta}{z (z-1)}&\end{eqnarray*}
the pullback 
${}_2F_1(\frac{a}{2}-\frac{1}{6},-\frac{a}{2}+\frac{1}{3}, \frac{2}{3},j(z)),$
is a solution of (see Remark~\ref{trafo} a))
\begin{eqnarray*}  y''+ p_1(z)y'+p_2(z) y=0, \end{eqnarray*}
\begin{eqnarray*}  p_1(z)=\frac{7z^2-21z+18}{3 z^3-14 z^2+27 z}, &&p_2(z)=\frac{-16(3a-1)(3a-2)(z^3-6z^2+15z-12)}{(3z^2-14z+27)^2z}.\end{eqnarray*}
A solution multiplied by $f(z)=(3z^2-14z+27)^{-1/3+a/2}$ gives us a Heun equation (see Remark~\ref{trafo} b)):
\begin{eqnarray*} 
 y''+ (\frac{(\frac{3}{2}-a)(6z-14)}{3z^2-14z+27}+\frac{2}{3z})y'+
  \frac{3 (9a-2)(-15a+10)z+2(3a-2)(96a-25)}{9z(3z^2-14z+27)}   y=0
\end{eqnarray*}
Our entries in Table 1, row 7, are obtained
via a M\"obius transformation to get the singularities at $0, 1, t, \infty $.
\end{exam}
The reason why this procedure always provides Heun equations will
be clear in the next section.

\section{Belyi functions}
\label{belyisection}
In order to derive further Heun-Picard-Fuchs equations
which can be  not necessarily obtained  from Herfurtner's list
we consider 
in this section  pull-backs of hypergeometric functions by rational Belyi
functions  with restricted ramification data.
These give rise to  second order differential equations without apparent singularities and in particular Heun equations.

\begin{prop}\label{Belyi}
 Let $j_1(z), j_2(z)\in \CC[z]$ be polynomials such that 
 $j(z)= \frac{j_1(z)}{j_2(z)}\in \CC(z)$ is a rational Belyi function unramified outside
 $\{ 0,1, \infty\}$.
 \begin{enumerate}
\item[a)] We can assume that the factorization is of the form
$$ j_1(z)= A \prod_{i\in I} (z-t_i)^{a_i},\;A \in \CC^{\ast}, 
   \quad  j_2(z)=\prod_{k\in K} (z-u_k)^{c_k}, \quad
   j_1(z)-j_2(z)=A\prod_{j\in J} (z-s_j)^{b_j},$$
 where $N:=\deg(j_1)>M:=\deg(j_2)$ and $(j_1(z),j_2(z))=1.$
\item[b)] Further for $$\Lambda=\prod_{\{t \in  \CC \mid (j_1j_2(j_1-j_2))(t)=0\}} (z-t)$$ we have $\deg(\Lambda)=N+1$.   
\end{enumerate}
\end{prop}

\begin{proof}
\begin{enumerate}
\item[a)]
Via a M\"obius-transformation and scaling we can assume that
\[ j(z)=\frac{j_1(z)}{j_2(z)}, \quad \deg (j_1(z))> \deg(j_2(z)). \]
\item[b)] 
Since $j(z)$ is only ramified at $0,1,$ and $\infty$ the Riemann-Hurwitz
formula implies that
\begin{eqnarray*}
 2N-2&=&\sum_i (a_i-1)+\sum_j (b_j-1)+\sum_k (c_k-1)+( N-\deg(j_2(z))-1)\\
\end{eqnarray*}
Hence
$deg(\Lambda) = N+1.$

\end{enumerate}
\end{proof}

\begin{theo}\label{Riemann}
 Let $j(z)$ be a rational Belyi function as in Proposition~\ref{Belyi}. 
Then 
\[ j_2(z)^{-\alpha}\cdot {}_2F_1(\alpha,\beta,\gamma,j(z))\]
satisfies $y''+q_1(z)y'+q_2(z)y=0,$ where
\begin{eqnarray*}
 q_1(z)=& \frac{\Lambda'}{\Lambda}+(\gamma-1) \frac{j_1'(z)}{j_1(z)}+(-\gamma+\alpha+\beta)\frac{(j_1(z)-j_2(z))'}{j_1(z)-j_2(z)}+(\alpha-\beta)\frac{j_2'(z)}{j_2(z)}&\\
q_2(z)=& \alpha \beta  \frac{j_1'(z)}{j_1(z)} \frac{(j_1(z)-j_2(z))'}{j_1(z)-j_2(z)}- \\
    &  \alpha \frac{j_2'(z)}{j_2(z)} \cdot\left (\frac{j_2'(z)}{j_2(z)}-
 \frac{\Lambda'}{\Lambda}-(\gamma-\beta-1) \frac{j_1'(z)}{j_1(z)}+(\gamma-\alpha)\frac{(j_1(z)-j_2(z))'}{j_1(z)-j_2(z)}-\frac{j_2''(z)}{j_2'(z)}  \right )&
\end{eqnarray*}
with the following Riemann scheme:

\[ \left(\begin{array}{ccccc}
     t_i & s_j &  u_k  &\infty \\
       0 &  0 &  0 & \alpha N \\
     (1-\gamma) a_i  & (\gamma -\alpha-\beta)b_j &(\beta-\alpha) c_k & \beta (N-M) +M \alpha
    \end{array}\right).\]

\end{theo}

\begin{proof}

By Remark~\ref{trafo}(a) the pull-back of the hypergeometric function 
$ {}_2F_1(\alpha,\beta,\gamma,j(z))$
satisfies
\begin{equation}\label{pb1}  y''+ (p_1(j(z)) j'(z)-\frac{j''(z)}{j'(z)})y'+p_2(j(z)) j'(z)^2 y =0, 
\end{equation}
\begin{equation*}  p_1(z)=\frac{\gamma-(\alpha+\beta +1)z}{z (1-z)}=\frac{\gamma}{z}+\frac{-\gamma+(\alpha+\beta +1)}{z-1}, \quad p_2(z)= \frac{\alpha \beta}{z (z-1)}.
\end{equation*}
By considering the exponents at the singularities locally one gets  
the following Riemann scheme:

\[ \left(\begin{array}{ccccc}
     t_i & s_j &  u_k  &\infty \\
       0 &  0 &  \alpha c_k & \alpha (N-M) \\
     (1-\gamma) a_i  & (\gamma -\alpha-\beta)b_j &\beta c_k & \beta (N-M)
    \end{array}\right).\]
Let $a_1(z)$ be the coefficient of $y'$. 
Then sum of the exponents at a finite singularity $t$ is given by  $1-\Res_{t}  (a_1(z))$, cf. \cite[Sec. 1.4]{iwa91}.
Thus together with Remark~\ref{trafo} the pull-back \eqref{pb1} satisfies
\begin{eqnarray*} y''+(\frac{\Lambda'}{\Lambda}+(\gamma-1) \frac{j_1'(z)}{j_1(z)}+(-\gamma+\alpha+\beta)\frac{(j_1(z)-j_2(z))'}{j_1(z)-j_2(z)}+(-\alpha-\beta)\frac{j_2'(z)}{j_2(z)}) y'+\\
 \alpha \beta ( \frac{j_1'(z)}{j_1(z)}-\frac{j_2'(z)}{j_2(z)})  \cdot 
 (\frac{(j_1(z)-j_2(z))'}{j_1(z)-j_2(z)}-\frac{j_2'(z)}{j_2(z)})y&=&0
\end{eqnarray*}
The solution multiplied by $f(z)=\prod (z-u_k)^{-\alpha c_k}$ 
satisfies
 $y''+q_1(z)y'+q_2(z)y=0$ with Riemann scheme
\[ \left(\begin{array}{ccccc}
     t_i & s_j &  u_k  &\infty \\
       0 &  0 &  0 & \alpha N \\
     (1-\gamma) a_i  & (\gamma -\alpha-\beta)b_j &(\beta-\alpha) c_k & \beta (N-M) +M \alpha
    \end{array}\right).\]
Again as above we can determine
$q_1(z)$ and $q_2(z)$ is obtained by using Remark~\ref{trafo}
\begin{eqnarray*}
 q_1(z)=& 
\frac{\Lambda'}{\Lambda}+(\gamma-1) \frac{j_1'(z)}{j_1(z)}+(-\gamma+\alpha+\beta)\frac{(j_1(z)-j_2(z))'}{j_1(z)-j_2(z)}+(\alpha-\beta)\frac{j_2'(z)}{j_2(z)}&\\
q_2(z)=& \alpha \beta ( \frac{j_1'(z)}{j_1(z)}-\frac{j_2'(z)}{j_2(z)})  \cdot 
 (\frac{(j_1(z)-j_2(z))'}{j_1(z)-j_2(z)}-\frac{j_2'(z)}{j_2(z)})-\alpha\frac{j_2'(z)}{j_2(z)} \cdot (2 (-\alpha)\frac{j_2'(z)}{j_2(z)}-(-(\alpha+1)\frac{j_2'(z)}{j_2(z)}+ \frac{j_2''(z)}{j_2'(z)}) &\\
 &     
 - (\frac{\Lambda'}{\Lambda}+(\gamma-1) \frac{j_1'(z)}{j_1(z)}+(-\gamma+\alpha+\beta)\frac{(j_1(z)-j_2(z))'}{j_1(z)-j_2(z)}+(-\alpha-\beta)\frac{j_2'(z)}{j_2(z)})).& 
\end{eqnarray*}
Simplifying the expression for $q_2(z)$ we get
\[q_2(z)=\alpha \beta  \frac{j_1'(z)}{j_1(z)} \frac{(j_1(z)-j_2(z))'}{j_1(z)-j_2(z)}- \]
\[  \alpha \frac{j_2'(z)}{j_2(z)} \cdot \left (   \frac{j_2'(z)}{j_2(z)}-
 (\frac{\Lambda'}{\Lambda}+(\gamma-\beta-1) \frac{j_1'(z)}{j_1(z)}+(-\gamma+\alpha)\frac{(j_1(z)-j_2(z))'}{j_1(z)-j_2(z)})-\frac{j_2''(z)}{j_2'(z)}\right) .\]

\end{proof}

\begin{cor}
Let
\[ 1-\gamma=\frac{1}{A}, \quad -\gamma+\alpha+\beta=\frac{1}{B}, \quad  \alpha-\beta=\frac{1}{C} ,\ \  A,B,C \in \NN_{\infty}.\]
If for the ramification indices of a rational Belyi function $j(z)$
the following
conditions hold
\begin{equation*}
 (*) \quad\quad\quad A \mid a_i \Rightarrow a_i=A,\quad  B \mid b_j \Rightarrow b_j=B, \ \  \  C \mid c_k \Rightarrow c_k=C 
\end{equation*}
then
\[ j_2(z)^{-\alpha}\cdot {}_2F_1(\alpha,\beta,\gamma,j(z))\]
satisfies a second order differential equation 
$y''+q_1(z)y'+q_2(z)y=0$
without apparent singularities.
\end{cor}

\begin{cor}
\label{8jan09}
Let $A,B,C \in \NN_{\infty}$
and $j(z)$ be a rational Belyi function satisfying
the conditions $(*).$
Let also $4= \# \{ a_i \mid a_i \neq A\} +\# \{ b_j\mid b_j \neq B\} 
+\#\{ c_k \mid c_k \neq C\} .$
Then we have
\begin{enumerate}
\item[a)]
The following function
\[ j_2^{-\alpha}(z) \cdot {}_2F_1(\alpha,\beta,\gamma,j(z)), \quad 1-\gamma=\frac{1}{A}, \quad -\gamma+\alpha+\beta=\frac{1}{B}, \quad  \alpha-\beta=\frac{1}{C}\]
satisfies a Heun equation.
\item[b)]
If $\{ c_k \mid c_k=C\} =\emptyset $ we get a one parameter
family of Heun equations corresponding to:
\[ j_2^{-\alpha}(z)\cdot  {}_2F_1(\alpha,\beta,\gamma,j(z)), \quad 1-\gamma=\frac{1}{A}, \quad -\gamma+\alpha+\beta=\frac{1}{B}\]
\item[c)]
If 
 $\{ c_k \mid c_k=C\} =\emptyset $ and $\{ b_k \mid b_k=B\} =\emptyset $  we get a two parameter
family of Heun equations corresponding to:
\[ j_2^{-\alpha}(z)\cdot  {}_2F_1(\alpha,\beta,\gamma,j(z)), \quad 1-\gamma=\frac{1}{A}\]
\item[d)]
If 
 $\{ c_k \mid c_k=C\} =\emptyset $, $\{ b_k \mid b_k=B\} =\emptyset $
 and $\{ a_i \mid a_i=A\} =\emptyset $ we get a three parameter
family of Heun equations corresponding to: 
\[ j_2^{-\alpha} (z)\cdot {}_2F_1(\alpha,\beta,\gamma,j(z)).\]
\end{enumerate}
\end{cor}

As noted by one of the anonymous referees there is an  elegant way
to derive Corollaries 1 and 2 by using simple local analysis
and avoiding Theorem~1.
However for future reference and in order to have an explicit formula
for the differential equations we keep it despite the cumbersome
calcuations appearing in its proof. 
Furthermore the computation of the Heun equation $L$ follows from a two
term local expansion of $L(f(z)$, where $f(z)$ is a known solution, e.g. 
at $z=0$.

\begin{rem}\rm
%\begin{enumerate}
%\item 
In order to prove Corollary~1 we at first determine the local exponents
      of the second order differential equation $L$ satisfied by
       \[ j_2(z)^{-\alpha}\cdot {}_2F_1(\alpha,\beta,\gamma,j(z)).\]
      We assume that $j(\infty)=\infty$ and concede that $\infty$ is going to
      be a singular point of $L$.
      For any finite point $z_0$ we can make the following
      observation. Suppose
      that $j(z_0) \not \in \{0,1,\infty\}$. Then $j'(z_0) \neq 0$ and
      $L$ will have the same local exponents as the hypergeometric equation
      at the point $j(z_0)$, hence $0, 1.$
      So $L$ does not have a singularity at $z_0$.
      Suppose $j(z_0)=0$. Let $a$ be the zero multiplicity of $z_0$.
      Then we have locally $j(z)\sim (z-z_0)^a$ and the local exponents of
      the hypergeometric equation are multiplied by $a$.
      Namely, if $f(x)=1+O(x)$ and $x^{1-\gamma}g(x)$ with $g(x)=1+O(x)$ are
      local solutions of the hypergeometric equation, clearly
      $f(j(z))=1+O(z-z_0)^a$ and
      $j(z)^{1-\gamma}g(j(z))=(z-z_0)^{a(1-\gamma)}(1+O(z-z_0)$ are the local
        solution of $L$ around $z_0$.
      So the local exponents of $L$ at $z_0$ read $0, (1-\gamma)a.$
      When $(1-\gamma)a=1$ it is obvious
      that $L$ does not have a singular point there.
      If at a point $z_0$ we have local exponents $0, 1$ and a basis of
      holomorphic
      solution the point $z_0$ is not singular.
      Of course, when $(1-\gamma)a$ is an integer $\neq 1$ we have an apparent
      singularity we may not get rid of. But this is excluded by condition
      $(*)$ in Corollary~1.
      Similarly we proceed with the cases $j(z_0)=1, \infty.$
      Corollary~1 is now also immediate.
         
%\item 
      The computation of the Heun equation is then straightforward.
      We know all local exponents, all that is needed is the value of the
      accessory parameter $q.$ 
      But this can be  computed by applying the Heun operator $L$ with the
      unknown parameter $q$ to a known local solution $f(z)$ at $z=0$, say.
      Then $q$ can be solved by consideration of the first two terms of the
      local expansion of $Lf=0$ in $z$ (or some other point).
%\end{enumerate}
\end{rem}

\begin{rem}\label{23param}\rm
 Herfurtner has classified all rational $j(z)$-functions
 such that  $4= \# \{ a_i \neq 2\}+\# \{ b_j \neq 3\} 
+\# \{ c_k \} .$
 Thus we always obtain at least a $1$ parameter family of Heun equations. 
Note that our $a,b,c$ notation in Table 1 refers to the
notations introduced in \S\ref{algorithm}.
If we are in the one parameter case  in Table 1 ($a=b=c$) then the relation is  $\alpha=\frac{a}{2}+\frac{1}{3},\;
 \gamma=\frac{2}{3}, \;\beta=-\frac{1}{6}+\frac{a}{2}$.

Next we list also the $2$ and $3$-parameter families of Heun equation and
 the corresponding $j(z)$-functions satisfying the hypothesis of Corollary
 \ref{8jan09}, part c,d. 
 It just an easy consequence of the Hurwitz formula that we have computed
 all $2$-and $3$-parameter families of Heun equations in Table 1b:

 Let $G$ be a Gauss hypergeometric differential equation
and $j(z)$ be a rational function such that
the pullback of $G$ with $j(z)$ of degree $N$ gives rise to a Heun equation.
Let  $a$, $b$ and $c$ denote the orders of the  local (projective) monodromy of $G$ at the singularities $0, 1$ and $\infty.$
We get the following conditions for the ramification:
Over $0$ we have $r+r_1$ points, where at the first $r$ points  have trivial
monodromy and the last $r_1$ points  have local monodromy of order dividing $a$.
This
can be written as 
 \begin{eqnarray} \label{eq1}(ax_1,\ldots, ax_r,\alpha_1,\ldots,\alpha_{r_1}),
    \end{eqnarray}
 where $r,r_1\in \NN_0, x_i \in \NN,
   a \nmid \alpha_i$.
Note that the sum over all ramification orders is $N.$
Similarly we get the corresponding ramification  over $1$ and $\infty:$
 \begin{eqnarray}   \label{eq2}    (by_1,\ldots, by_s,\beta_1,\ldots,\beta_{s_1}),&
       (cz_1,\ldots, cz_t,\gamma_1,\ldots,\gamma_{t_1}), \end{eqnarray} where $s,t,s_1,t_1 \in \NN_0, y_j,z_k \in \NN,
  b \nmid \beta_j, c \nmid \gamma_k$.
Since a Heun equation  has only 4 singularities we get that $r_1+s_1+t_1=4$.
The  Riemann-Hurwitz formula implies that
  \[ -1 \geq (-N) + \frac{1}{2}( (N-(r+r_1))+(N-(s+s_1))+(N-(t+t_1))=
              \frac{1}{2}(N-(r+s+t+4)).  \]
Hence  \begin{eqnarray}\label{eq3} N \leq r+s+t+2.\end{eqnarray}
To obtain a $2$- or $3$- parameter family we can assume $s=t=0.$ 
Thus $N\leq r+2\leq \frac{N}{2}+2$ which implies $N \leq 4.$
The following table follows from a classification of all
triples $(g_1,g_2,g_3),\; g_1g_2g_3=id,\;g_i \in S_N,$ where the
elements $g_i$ have the prescribed cycle decomposition and the
construction of the corresponding Belyi-function.\\

{\tiny
Table 1b: 2-and 3-parameter families of geometric Heun equations, $\alpha,\beta,\gamma \in\Q$
\[
%\begin{array}{|c|c|c|c|c|c|c|c|c|}
%\multicolumn{11}{c}{$***$}
\begin{array}{|c|c|c|c|c|c|c|c|}
\hline
*  &   q & t & \th_1 & \th_2 & \th_3 &
% \th_4 &
 \th_{42} & \th_{41}
  \\ \hline
\hline
I_3\; I_1\; II\; I_0^* &
   8\alpha(-3 \gamma+2) &
   -8 &
   1-\gamma &
   3(1- \gamma) &
    \gamma -2 \alpha-\frac{1}{2} &
%   3(\gamma-2 \alpha-\frac{1}{2}) &
   -2 \alpha+3 \gamma-\frac{3}{2} &
   4 \alpha 
 \\ \hline
31&
   0 &
   -1 &
   2(1- \gamma) &
   -4 \alpha+2 \gamma-1 &
   2(1- \gamma) &
 %  -4 \alpha+2 \gamma-1 &
   2 \gamma-1 &
   4 \alpha 
 \\ \hline
32 &
   0 &
   -1 &
   -\alpha-\beta+\gamma &
   2(1- \gamma) &
   -\alpha-\beta+\gamma &
 %  2(\beta- \alpha) &
   2 \beta &
   2 \alpha 
 \\ \hline
33&
   -3 \alpha \beta t_1 &
   {\frac{t_1^2}{3},\  t_1^2+3t_1+3=0}  & % t=zeta_6
 \frac{2}{3}-(\alpha+\beta)  &
 \frac{2}{3}-(\alpha+\beta)  &
 \frac{2}{3}-(\alpha+\beta)  &
%   3( \beta- \alpha) &
3 \beta &
3 \alpha 
 \\ \hline
%34a) &
% (\alpha-\frac{1}{2})(3\gamma-2)&
%   -\frac{1}{3} &
%   -2 \alpha+\gamma+\frac{1}{2} &
%  3 (1- \gamma) &
%   \frac{1}{2} &
%   2(-2 \alpha+\gamma+\frac{1}{2}) &
%   (-\alpha+2 \gamma-\frac{1}{2}) &
%   (3 \alpha-\frac{3}{2}) 
34&
  0&
   -\frac{1}{3}&
   \frac{1}{2}&
   2(1-\gamma)&
   1-\gamma&
%   -6\alpha+3\gamma+\frac{3}{2}&
   -3\alpha+3\gamma&
   3\alpha-\frac{3}{2}
 \\ \hline 
%  2(-\alpha+\beta)&
%   -\alpha+\beta&
%   -3\alpha-3\beta+\frac{3}{2}&
%(-3 \beta+\frac{3}{2})&
%3 \alpha
%\\ \hline
I_2\;I_1\;III\; I_0^*&
   6\alpha (2 \beta-2 \alpha-1)&
   -8&
   -2\alpha-2 \beta +\frac{4}{3}&
   -2 \alpha+2 \beta&
   -\alpha-\beta+\frac{2}{3}&
%   -\alpha+\beta&
   2 \alpha+\beta&
   3\alpha
\\ \hline
35 &
   0&
   -1 &
   1-\gamma &
   2(1-\gamma) &
   1-\gamma &
 %  2(1-4 \alpha+2 \gamma) &
   4(\gamma- \alpha) &
   4 \alpha-2 
\\ \hline
\end{array}\]
%\end{center}
%\end{sidewaystable}
%}
}

\[
\begin{array}{|c|c|c|c|}
\hline
  &  j(z) & \mbox{ramification data} & (\alpha,\beta,\gamma) 
  \\ \hline
\hline
I_3\; I_1\; II\; I_0^* & \frac{(z^4+8z^3)}{(64 z-64)}&  (3)(1), \; (2)(2),\; (3)(1) & \beta=\gamma-\alpha-\frac{1}{2} \\
31 &\frac{(-z^4+2 z^2-1)}{(4 z^2)}&  (2)(2), \; (2)(2),\; (2)(2) & \beta=\gamma-\alpha-\frac{1}{2} \\
32 & z^2 & (2),\; (1)(1), \; (2) & \\
33 & (z+1)^3 & (3),\; (1)(1)(1), \; (3) & \gamma=\frac{2}{3} \\
34&\frac{1}{4} z^2 (z+3)&  (2)(1),\; (2)(1),\;(3) & \beta=\gamma-\alpha-\frac{1}{2} \\ 
%34a)& \frac{(4 z^3)}{(3 z+1)}& (2)(1),\; (2)(1),\; (3) &\beta=\alpha-\frac{1}{2}\\
I_2\;I_1\;III\; I_0^* & -\frac{1}{27} \frac{(z-4)^3}{z^2} & (3),\;(2)(1),\; (2)(1) & \gamma=\frac{2}{3} \\
35 & -4z^2(z-1)(z+1) & (2)(1)(1),\; (2)(2),\; (4) & \beta=\gamma-\alpha-\frac{1}{2} \\ \hline
\end{array}
\]
Note that all the $2$- and $3$-parameter Heun equations in Table 1 appear in Table 1b while 
the 2-parameter Heun equations in Table 1b, $I_3\; I_1\; II\; I_0^*$ and $33$ extend Table 1. 
(We use the notation $I_3\; I_1\; II\; I_0^*$ and $I_2\;I_1\;III\; I_0^*$
 to indicate that the $j$-function we use here is the
 same as in Herfurtner's list up to a M\"obius-transformation. In Herfurtner's
 list the corresponding differential equations are Gauss hypergeometric ones.)
\end{rem}

%----------------------------------
\section{Lam\'e equations}
\label{lame}
The most studied Heun equations are the so called Lam\'e equations: 
\[ p(z) \frac{d^2y}{dz^2}+\frac{1}{2} p'(z) \frac{dy}{dz}-(n(n+1)z+B)y=0, \]
where $p(z)=4z^3-g_2z-g_3$.
Hence we also  list the cases, where  the Heun equations in Table~1 specialize
(after M\"obius transformations) to Lam\'e equations.
 The relation between the above standard notation for a Lam\'e equation
and our notation is given by the transformation $z \mapsto z+\frac{t+1}{3},$
that maps our singularities at $t, 0, 1$ to $t-\frac{t+1}{3}, -\frac{t+1}{3},
1-\frac{t+1}{3}$
and we have
%\[  B=4 K t^2-4 K t-\frac{2}{3} t \th_4^2-\frac{4}{3} t \th_4-\frac{1}{2} t+\frac{1}{3} \th_4^2+\frac{2}{3} \th_4+\frac{1}{4}\]
%\[ n(n+1)=(\th_4+\frac{1}{2})(\th_4+\frac{3}{2}) \] %-4 \th_{4,1}\th_{4,2}\]
$$
n=\th_4-\frac{1}{2}, \quad B=4q
$$

%\newpage
{\tiny
\begin{center}
{Table 2: Lam\'e equations coming from geometry
\begin{tabular}{|c|c|c|c|c|c|c|c|}
%\multicolumn{11}{c}{$***$}
\hline
  &           &       &     &                 &            &            &      \\ 
* &conditions &   $q$ & $t$ & $\th_i,$ & $\th_4$  & $\th_{42}$ & $\th_{41}$  \\ 
  &           &       &     &    $i=1,2,3$               &            &            &      \\      \hline
\hline
1 &$a=1$ &
$ \frac{5}{3} t_1 $ &
$\frac{t_1^2}{3}, \ t_1^2+3t_1+3=0 $ &
$  \frac{1}{2}  $ &
$  \frac{9}{2}  $ &
$  \frac{5}{2}  $ &
$  -2  $
\\ \hline
5 &$a+c=\frac{3}{2},a+b=\frac{5}{4}$&
% (12z4-12z2+12  $ &
% (8z6-12z4-12z2+8  $ &
$  0$ &
$ -1 $ &
$  \frac{1}{2}  $ &
$  \frac{7}{2}-4a  $ &
$  2-2a  $ &
$  2a-\frac{3}{2}  $
\\ \hline
6 &$a=\frac{2}{3}$&
% (3z4+12z3+18z2+36z+27  $ &
% (-z6-6z5-15z4+45z2+54z+27  $ &
$ 0  $ &
$\frac{t_1^2}{3}, t_1^2+3t_1+3=0$ &
$  \frac{1}{2}  $ &
$  \frac{1}{2}  $ &
$  \frac{1}{2}  $ &
$  0  $
\\ \hline
11 &$a=1$ &
$  \frac{37}{128}  t_1 $ &
$\frac{t_1^2}{8}, 4t_1^2+13t_1+32=0$ &
$  \frac{1}{2}  $ &
$  \frac{7}{2}  $ &
$  2  $ &
$  -\frac{3}{2}  $
\\ \hline
12 &$b=\frac{3}{4},a=\frac{5}{12}$ & 
%  12z4-72z3+108z2-36z  $ &
%  8z6-72z5+216z4-252z3+108z2  $ &
$ \frac{1}{32} $ &
$ \frac{3}{4} $ &
$  \frac{1}{2}  $ &
$  \frac{1}{4}  $ &
$ \frac{3}{8} $ &
$  \frac{1}{8}  $
\\ \hline
14 &$b=\frac{2}{3},a=\frac{5}{6}$ &
%  48z4-48z3-9z2+6z+3  $ &
%  64z6-96z5+6z4+25z3-3z2+3z+1  $ &
$  \frac{1}{8}  $ &
$ -3 $ &
$ \frac{1}{2} $ &
$  1  $ &
$  \frac{3}{4}  $ &
$  -\frac{1}{4}  $
\\
&$b=\frac{2}{3},a=\frac{7}{12}$ &
$  \frac{3}{64}  $ &
$ \frac{3}{4} $ &
$ \frac{1}{2} $ &
$  \frac{1}{4}  $ &
$  \frac{3}{8}  $ &
$  \frac{1}{8}  $
\\ \hline
26 &$a=1$&
$  -\frac{29}{125}  t_1  $ &
$\frac{4t_1^2}{125},\ t_1^2-11t_1+\frac{125}{4}=0$ &
$ \frac{1}{2} $ &
$  \frac{5}{2}  $ &
$  \frac{3}{2}  $ &
$  -1  $
\\
 &$a=\frac{3}{5}$&
$  \frac{1}{4t_1}   $ &
$\frac{125}{4t_1^2},\ t_1^2-11t_1+\frac{125}{4}=0$ &
$ \frac{1}{2} $ &
$  \frac{1}{10}  $ &
$  \frac{3}{10}  $ &
$  \frac{2}{10}  $
\\ \hline
27 &$b=\frac{3}{4}$& 
$  0  $ &
$ -1 $ &
$ \frac{1}{2} $ &
$  2a-\frac{1}{2}  $ &
$  a  $ &
$  -a+\frac{1}{2}  $
\\ \hline
28 & $a=\frac{2}{3}$ &
$  0  $ &
$-\frac{t_1^2}{3},\ t_1^2-6t_1-3=0$ &
$  \frac{1}{2}  $ &
$ \frac{1}{2} $ &
$  \frac{1}{2}  $ &
$  0  $
\\ \hline

32& $a=\frac{3}{4}$ & 
$  0  $ &
$ -1 $ &
$ \frac{1}{2} $ &
$ \frac{2}{3}  $ &
$  \frac{7}{12}  $ &
$  -\frac{1}{12}  $
\\ \hline
33 &&
%$ (\frac{3}{2}a^2-\frac{3}{2}a+\frac{1}{3}) t_1+ (\frac{9}{4}a^2-\frac{9}{4}a+\frac{1}{2}) $ &
$\frac{1}{12}(3a-1)(3a-2)t_1 $&
$\frac{t_1^2}{3},\ t_1^2+3t_1+3=0$ &
$ \frac{1}{2} $ &
$ 3a-\frac{3}{2} $ &
$ \frac{3}{2}a-\frac{1}{2} $ &
$ 1-\frac{3}{2}a $
\\ \hline
36& $a=\frac{5}{8}$&
$  \frac{1}{216}  t_1 $ &
$\frac{t_1^2}{27},\ t_1^2-10t_1+27=0$ &
$ \frac{1}{2} $ &
$ \frac{1}{3}  $ &
$  \frac{5}{12}  $ &
$  \frac{1}{12}  $
\\ \hline
\end{tabular}}
\end{center}
%\end{sidewaystable}
}

%-------------
\section{Comparison with known results}
\label{comparison}
 The Heun equations computed  by Maier, s. \cite[Thm. 3.8] {Maier05}, via polynomial pull-backs of hypergeometric differential equations appear with the exception
 of  the equations $(3.6.a)$ in our list:
 \[ \begin{array}{c|c}
      \mbox{Nr. in Maier's list} & \mbox{ Nr. in List~1} \\
   \hline
          (3.5a)    & 5 \\
          (3.5b)    & 12\\
          (3.5c)    & 2 \\
          (3.6b)    & 36\\
          (3.6c)&  37\\
          (3.6d) & 38
    \end{array} \]
 The Heun equations  $(3.6.a),$
 \[ y''+ \frac{2-(\alpha +\beta)}{3}(\frac{1}{z+\zeta_3}+\frac{1}{z}+\frac{1}{z-1})y'+\frac{\alpha\beta z- \frac{\alpha\beta}{3} (1-\zeta_3)}{z(z-1)(z+\zeta_3)}y=0, \quad \zeta_3^3=1
 \]
 which depend on two free parameters $\alpha$ and $\beta$,
 appears in Table 1b), row 33.
 If the monodromy group is contained in $\SL_2(\ZZ)$ (for instance for the parameters
 $a, b, c $ in the Table~1 equal to $\frac{1}{2}$)
 then some of these differential equations appear also in literature, e.g.
% Ap\'ery's proof of irrationality of $\zeta(3)$  (see
% \cite[Chap. 2]{Chud1383}) or 
in the study of the Grothendieck $p$-curvature conjecture
\cite{Beukers02} and \cite{Chud1383}.

 We list these examples, where we have used  M\"obius transformations to obtain
 coefficients in $\QQ[z].$
\[ \begin{array}{cllll}
  \mbox{Table 1} &\mbox{Notation } &  & p(z)\\
                       & \mbox{in \cite{her91}} &&\\
  1 &I_1 I_1 I_1 I_9 & p(z)y''+p'(z)y'+(z+\frac{1}{3})y=0 & z(z^2+z+\frac{1}{3}) &\\
  2 &I_1 I_1 I_2 I_8 &  p(z)y''+p'(z)y'+zy=0 &  z(z-1)(z+1) \\
  3 &I_1 I_2 I_3 I_6&  p(z)y''+p'(z)y'+(z-\frac{1}{4})y=0& z(z-1)(z+\frac{1}{8})\\
  4& I_1 I_1 I_5 I_5& p(z)y''+p'(z)y'+(z+3)y=0  &z(z^2+11z-1) \\
 \end{array}\]
Note that these examples also arise from Beauville's list, see
\cite{Beauville}.

\begin{rem}\rm
At the time this paper was accepted 
further research was meanwhile carried out by other authors. 
All hypergeometric to Heun transformations with two or three continuous
parameters, cf. Remark~\ref{23param}, up to M\"obius transformations were independently studied by Filipuk and Vidunas
in \cite{FV09}.
Also there is a 
{ \it Table of all hyperbolic 4-to-3 rational Belyi maps and their dessins}
available by van Hoeij and Vidunas \cite{HV2011}, were all Belyi maps satisfying
the conditions formulated in Corollaries 1 and 2 are listed together
with further interesting properties.
We also would like to point out that Sijsling has classified as a part of his
PhD thesis all Lam\'e equations with arithmetic monodromy group of type
$(1,e)$ that are pullbacks of hypergeometric differential equations,
see \cite{Si2012}.
\end{rem}

%%%%%%%%%%%%%%%%%%%%%%%%%%%%%%%%%%%%%%%%%%%%%%%%%%%55

%----------------------------------------------------------------------------------
{}

\begin{thebibliography}{1}

\bibitem{and88}
Y. Andr\'e, $G$-functions and geometry,
Aspects of Mathematics, E13. Friedr. Vieweg \& Sohn, Braunschweig, 1989.


\bibitem{Beauville}
A. Beauville,
Les familles stables de courbes elliptiques sur $P\sp{1}$ admettant quatre fibres singuli\'eres. 
{\em C. R. Acad. Sci. Paris.} 294 (1982), no. 19, 657--660.


\bibitem{ben05}
B. Ben~Hamed and L. Gavrilov,
\newblock Families of {P}ainlev\'e {VI} equations having a common solution.
\newblock {\em Int. Math. Res. Not.}, (60):3727--3752, 2005.


\bibitem{Beukers02}
F. Beukers,
\newblock On Dwork's accessory parameter problem. 
\newblock {\em Math. Z.} 241 (2002), no. 2, 425--444. 


%\bibitem{bewa04}
%F. Beukers; A. van der Waall.
%Lam\'e equations with algebraic solutions. 
%J. Differential Equations 197 (2004), no. 1, 1--25. 

\bibitem{Chud1383}
D. V. Chudnovsky and  G. V. Chudnovsky, Computational problems in arithmetic of linear differential equations. Some Diophantine applications.  Number theory (New York, 1985/1988),  12--49, Lecture Notes in Math., 1383, Springer, Berlin, 1989.


\bibitem{dor01}
C.~F. Doran,
\newblock Algebraic and geometric isomonodromic deformations.
\newblock {\em J. Differential Geom.}, 59 (1):33--85, (2001).


\bibitem{Doran98}
C.~F. Doran,
\newblock Picard-Fuchs uniformization: modularity of the mirror map and
mirror-moonshine. 
\newblock {The arithmetic and geometry of algebraic cycles} (Banff, AB, 1998), 257–281, CRM Proc. Lecture Notes, 24, Amer. Math. Soc., Providence, RI, 2000.


\bibitem{GPS01}
G.-M. Greuel, G.~Pfister, and H.~Sch\"onemann,
\newblock {\sc Singular} 2.0.
\newblock {A Computer Algebra System for Polynomial Computations}, Centre for
  Computer Algebra, University of Kaiserslautern, 2001.
\newblock {\tt http://www.singular.uni-kl.de}.

\bibitem{FV09}
 G. Filipuk and R. Vidunas, 
\newblock{General transformations between the Heun and Gauss hypergeometric functions.}
\newblock arXiv:math/0910.3087 (2009).


\bibitem{her91}
S. Herfurtner,
\newblock Elliptic surfaces with four singular fibres.
\newblock {\em Math. Ann.}, 291(2):319--342, 1991.

\bibitem{iwa91}
K. Iwasaki, H. Kimura, S. Shimomura, and M. Yoshida,
\newblock {\em From {G}auss to {P}ainlev\'e}.
\newblock Aspects of Mathematics, E16. Friedr. Vieweg \& Sohn, Braunschweig,
  1991.
\newblock A modern theory of special functions.

\bibitem{ka70}
M. N. Katz,
Nilpotent connections and the monodromy theorem: Applications of a result of Turrittin.
{\em Inst. Hautes \'Etudes Sci. Publ. Math.} No. 39 (1970), 175--232. 

\bibitem{Krammer96}
 D. Krammer.
\newblock An example of an arithmetic Fuchsian group.  
\newblock {\em J. Reine Angew. Math.}  473  (1996), 69--85. 

\bibitem{Maier05}
R. S. Maier,
\newblock On reducing the Heun equation to the hypergeometric equation.
\newblock {\em J. Differential Equations} 213 (2005), no. 1, 171--203. 

\bibitem{mov08}
H.~Movasati,
\newblock On {R}amanujan relations between {E}isenstein series.
\newblock {\em To appear in Manuscripta Mathematica}, 2012.



\bibitem{more}
H.~Movasati, S. Reiter, 
\newblock Painlev\'e VI equations with algebraic solutions and family of curves, 
\newblock {\em  Journal of Experimental Mathematics}, Vol. 19, Number 2  (2010), 161-173.

%\bibitem{more2}
%H.~Movasati, S. Reiter, Library {\tt painleve-heun.lib} written is {\sc  Singular}, {\tt www.impa.br/$\sim$ hossein}. 


\bibitem{R08}
 S. Reiter,
\newblock{Halphen's transform and middle convolution.}
\newblock{arXiv:math/0903.3654 (2009)}.

\bibitem{Si2012}
 J. Sijsling,
\newblock{Arithmetic $ (1;e)$-curves and Belyi maps.}
\newblock{\em Math. Comp.}, 81, (2012), 1823-1855. 



\bibitem{Sti83}
P. F. Stiller, 
\newblock On the uniformization of certain curves.
\newblock {\em Pacific J. Math.} 107 (1983), no. 1, 229-244. 

\bibitem{HV2011}
M. van Hoeij and R. Vidunas,
\newblock{Table of all hyperbolic 4-to-3 rational Belyi maps and their
  dessins.}
\newblock{\tt http://www.math.fsu.edu/~hoeij/Heun/overview.html} (2011).
\end{thebibliography}
\end{document}